\documentclass[12pt]{article}

\usepackage{epsfig}
\usepackage{epstopdf}
\usepackage{graphicx}
\usepackage{caption}
\usepackage{subcaption}
\usepackage{graphics}
\usepackage{amsthm}
\usepackage{verbatim}
\usepackage{amsmath,amssymb,xspace}
\usepackage{pdfsync}
\usepackage{appendix}

\newtheorem{theorem}{Theorem}

\newtheorem{lemma}{Lemma}
\newtheorem{prop}{Proposition}

\newtheorem{cor}{Corollary}
\newtheorem{defi}{Definition}

\let \leq \leqslant
\let \geq \geqslant

\let \epsilon \varepsilon
\let \hat \widehat

{ \par \medskip \par
  \noindent \textit{\textbf{Demonstration\/}} : }{\null \hfill $\Box$ \par }

\newcommand\pdiff[2]{\ensuremath{\frac{\partial #1}{\partial #2}}}
\newcommand\diff[2]{\ensuremath{\frac{d #1}{d #2}}}

\newcommand{\be}{\begin{equation}}
\newcommand{\ee}{\end{equation}}
\newcommand{\ben}{\begin{equation*}}
\newcommand{\een}{\end{equation*}}

\newcommand{\bery}{\begin{eqnarray}}
\newcommand{\eery}{\end{eqnarray}}

\title{Spatial and Modal Superconvergence of the Discontinuous Galerkin Method for Linear Equations}

\author{N.~Chalmers\footnote{Email: bnachalm@uwaterloo.ca} and L.~Krivodonova\footnote{Email: lgk@uwaterloo.ca Corresponding author.}}
\date{}

\begin{document}
\maketitle
\begin{abstract}
We apply the discontinuous Galerkin finite element method with a degree $p$ polynomial basis to the linear advection equation and derive a PDE which the numerical solution solves exactly. We use a Fourier approach to derive polynomial solutions to this PDE and show that the polynomials are closely related to the $\frac{p}{p+1}$ Pad\'e approximant of the exponential function. We show that for a uniform mesh of $N$ elements there exist $(p+1)N$ independent polynomial solutions, $N$ of which can be viewed as physical and $pN$ as non-physical. We show that the accumulation error of the physical mode is of order $2p+1$. In contrast, the non-physical modes are damped out exponentially quickly. We use these results to present a simple proof of the superconvergence of the DG method on uniform grids as well as show a connection between spatial superconvergence and the superaccuracies in dissipation and dispersion errors of the scheme. Finally, we show that for a class of initial projections on a uniform mesh, the superconvergent points of the numerical error tend exponentially quickly towards the downwind based Radau points.
\end{abstract}



\section{Introduction}
In this paper we apply the discontinuous Galerkin (DG) method with a polynomial basis to the one-dimensional linear hyperbolic problem
\begin{equation} \label{eq:linear}
u_t + a u_x = 0
\end{equation}
with $a > 0$ constant, subject to periodic boundary conditions on interval $I$ and sufficiently smooth initial data $u(x,0)$. We derive a partial differential equation (PDE) on the $j$-th cell which is solved by the polynomial numerical solution $U_j$ exactly. This PDE is similar to the original advection equation with a forcing term. Then, by applying classical Fourier analysis, we find polynomial solutions to this PDE. These solutions are closely related to both the $(p+1)$-th right Radau polynomial $R_{p+1}^-$ and the $\frac{p}{p+1}$ Pad\'e approximant of the exponential function $e^z$, where $p$ is the degree of the polynomial approximation. We use these particular polynomial solutions to show that for a uniform computational mesh of $N$ elements there exist $(p+1)N$ independent polynomial solutions, $N$ of which can be seen as physical and $pN$ as non-physical. Furthermore, we prove a property of the superconvergent points of the DG method conjectured by Biswas et al \cite{BDF94} in 1994. Specifically, we show that for a class of initial projections on a uniform mesh the numerical error will tend to a superconvergent form with order $p+2$ at the right Radau points. Moreover, we establish conditions on the projection of the initial data for which this property holds. We also show that this superconvergent form has high-order accuracy in the moments of the solution, i.e., the $L^2$ projection of the numerical solution onto the $m$-th Legendre polynomial is order $2p+1-m$ accurate. In particular, the cell average of the solution is advected with order $2p+1$ accuracy. 

Superconvergence of the DG method for one-dimensional problems has been studied in several papers. Following the conjecture made by Biswas et al, Adjerid et al in \cite{adjerid01a} proved order $(p+2)$ convergence of the DG solution at the downwind-based Radau points and order $2p+1$ convergence at the downwind end of each cell for ODEs. An order $(p+ \frac{3}{2})$ convergence rate of the DG solution to a particular projection of the exact solution was later shown by Cheng and Shu in \cite{ChengShu08,ChengShu10}. Yang and Shu then showed the same superconvergence property with order $p+2$ convergence for linear hyperbolic equations in \cite{yangshu2012}. Fourier analysis of the DG solution has also been applied to the DG solution in order to investigate superconvergence by symbolically manipulating the discretetization matrices for low order ($p=1,2,$ and 3) approximations \cite{zhongshu2011,guo2013}. 

Likewise, the connection between the DG scheme and the Pad\'e approximants of the exponential function has been observed in several works. The stability region of the DG method for ODEs was demonstrated by Le Saint and Raviart \cite{LeSR74} to be given by $|R(\lambda h)| \leq 1$ where $R(z)$ is the $\frac{p}{p+1}$ Pad\'e approximant of $e^z$ and $h$ is the grid spacing. In \cite{HuAtkins02}, Hu and Atkins conjectured that certain polynomials involved in the analysis of the numerical dispersion relation are related to $\frac{p+1}{p}$ Pad\'{e} approximant of $\exp(z)$ and used this to show that the numerical dispersion relation is accurate to $(kh)^{2p+2}$, where $k$ is the wavenumber. This conjecture was proven and an extended analysis of the dispersion and dissipation errors was given by Ainsworth in \cite{Ainsworth04}. Later, a connection between the spectrum of the DG method on linear problems and the Pad\'e approximant of the exponential was investigated by Krivodonova in Qin in \cite{KQ13}. In this paper, we demonstrate that the numerical solutions of the DG method are themselves closely related to this Pad\'e approximant and furthermore both the superconvergent local errors and superaccurate errors in dissipation and dispersion of the method can be seen as resulting from the accuracy of this Pad\'e approximant. We also use these numerical solutions to present a new proof of spatial superconvergence of the DG method on uniform grids. This proof is simpler than existing ones \cite{ChengShu08,ChengShu10} and does not require the assumption that the initial projection is superconvergent at the beginning of the computation. Moreover, we show what condition the initial projection should satisfy for the error to become and remain superconvergent at the right Radau points after sufficient time. 

The remainder of this paper is organized as follows. In Section 2 we begin by deriving the DG method and use the formulation to write a PDE which is solved by the numerical solution $U_j$ on the $j$-th cell. The main body of this paper is in Section 3 where by decomposing the numerical solution into Fourier modes, we use this PDE to find particular numerical solutions and establish our main results. These results are then illustrated  numerically in Section 4. Finally, conclusions and thoughts on future work are discussed in Section 5. 

\section{DG Method}

Our goal in this section is to derive a partial differential equation for the numerical solution of the DG method. We begin with the linear conservation law \eqref{eq:linear}. The domain is discretized into mesh elements $I_j=[x_{j},x_{j+1}]$ of size $h_j = x_{j+1}-x_j$, $j=1,2,...,N$. The discontinuous Galerkin spatial discretization on cell $I_j$ is obtained by approximating the exact solution $u$ by a polynomial of degree $p$, $U_j$. Multiplying (\ref{eq:linear}) by a test function $V\in \mathcal P_p$ and integrating the result on $I_j$ we obtain
\begin{equation}\label{eq:dg}
\frac{d}{dt} \int_{x_{j}}^{x_{j+1}} U_j V\,dx + \int_{x_{j}}^{x_{j+1}} a\pdiff{U_j}{x} V \, dx = 0,\quad \forall V\in \mathcal P_p.
\end{equation}
Here, $\mathcal P_p $ is a finite dimensional space of polynomials of degree up to $p$. Transforming $[x_{j},x_{j+1}]$ to the canonical element $[-1,1]$ by a linear mapping
\begin{equation}\label{eq:map}
x(\xi)= \frac{x_{j}+x_{j+1}} 2 + \frac {h_j} 2 \xi
\end{equation}
yields
\begin{equation}\label{eq:err2}
 \frac d {dt} \int_{-1}^{1} U_j V\,d\xi + \frac{2}{h_j} \int_{-1}^{1} a\pdiff{U_j}{\xi} V \, d\xi = 0,\quad \forall V\in \mathcal P_p.
\end{equation}
Integrating the flux term by parts we obtain,
\[
\frac d {dt} \int_{-1}^{1} U_j V\,d\xi + \frac{2a}{h_j}\left[U_{j}(1)V(1) - U_{j-1}(1)V(-1)\right] - \frac{2a}{h_j}\int_{-1}^{1} U_j \pdiff{V}{\xi} \, d\xi = 0,\quad \forall V\in \mathcal P_p.
\]
Note that we have chosen an upwind flux so that the value of $U$ at the cell interface is chosen to be the value on the left side of the discontinuity. Integrating by parts once more we can write,
\begin{equation}
\frac d {dt} \int_{-1}^{1} U_j V\,d\xi  + \frac{2a}{h_j}\int_{-1}^{1} \pdiff{U_j}{\xi} V \, d\xi = - \frac{2a}{h_j}[[U_{j}]]V(-1),\quad \forall V\in \mathcal P_p.
\label{eq:err3}
\end{equation}
where $[[U_{j}]] = U_{j}(-1) - U_{j-1}(1)$ denotes the jump between the endpoints of the numerical solution at the interface of the $j$-th and $(j-1)$-th cells. Note that the second integral in this expression is entirely local, as opposed to the term in \eqref{eq:err2}.  Next, we choose the Legendre polynomials as the basis for the finite element space $\mathcal P_p$. Recall \cite{Abram65}, that the Legendre polynomials $P_k(\xi)$, $k=0,1,2,\dots$, form an orthogonal system on $[-1,1]$
\begin{equation}\label{eq:legendre}
\int_{-1}^{1} P_kP_i \,d\xi = \frac {2}{2k+1}\delta_{ki},
\end{equation}
where ${\delta_{ki}}$ is the Kroneker delta. With the chosen normalization (\ref{eq:legendre}), the values of the basis functions at the end points of the interval $[-1,1]$ are \cite{Abram65}
\begin{equation}\label{eq:legendre2}
P_k(1)=1, \qquad  P_k(-1)=(-1)^k.
\end{equation}
The numerical solution can be written in terms of this basis as
\begin{equation}\label{eq:sol}
U_j=\sum_{k=0}^p c_{jk}P_k,
\end{equation}
where $c_{jk}$ is a function of time $t$. Substituting (\ref{eq:sol}) into (\ref{eq:err3}), choosing
$V=P_k$, $k=0,1,\dots,p$,  and using (\ref{eq:legendre2}) and (\ref{eq:legendre}) results in the system of equations
\begin{equation}\label{eq:coef0}
\diff{}{t} c_{jk} +  \frac{(2k+1)a}{h_j}\int_{-1}^{1} \pdiff{U_j}{\xi} P_k \; d\xi=  -\frac{(2k+1)a}{h_j}(-1)^k [[U_j]],
\end{equation}
for $k=0,1,\dots,p$. This system of equations is enough to determine $U_j$ defined by \eqref{eq:sol}. However, since we are interested in the equation which $U_j$ itself satisfies we can reconstruct $U_j$ by multiplying each equation in the system by $P_k$ and summing over all $k$. We can thereby write the following exact expression for $\pdiff{}{t} U_j$, 
\begin{equation}\label{eq:coef1}
\pdiff{}{t} U_j + \frac{2a}{h_j}\sum_{k=0}^p \frac{2k+1}{2} \left(\int_{-1}^{1} \pdiff{U_j}{\xi} P_k \; d\xi \right)P_k =   - \frac{a}{h_j}[[U_{j}]]\left( \sum_{k=0}^p (-1)^k(2k+1)P_k\right).
\end{equation}
Since the Legendre polynomials are an orthogonal family it is clear that the first summed term in this expression is simply the projection of $\pdiff{U_j}{\xi}$ into the finite element space $V_h$. Moreover, since $\pdiff{U_j}{\xi}$ is already in the finite element space this projection is exact. Hence we can write, 
\begin{equation}\label{eq:coef2}
 \pdiff{}{t} U_j + \frac{2a}{h_j}\pdiff{}{\xi}U_j  = - \frac{a}{h_j}[[U_{j}]]\left( \sum_{k=0}^p (-1)^k(2k+1)P_k\right).
\end{equation}
To simplify this expression further let us use the following proposition:

\begin{prop}
\[
\sum_{k=0}^p (-1)^k(2k+1)P_k = (-1)^p\diff{}{\xi}R^-_{p+1}(\xi),
\]
where $R^-_{p+1}$ is the right Radau polynomial \cite{Abram65} of degree $p+1$,  which is defined as $R^-_{p+1} = P_{p+1}-P_p$. 
\end{prop}

\begin{proof}
It is known that the Legendre polynomials satisfy \cite{Abram65},
\[
\diff{}{\xi}P_{k+1} = (2k+1)P_k + (2k-3)P_{k-2} + (2k-7)P_{k-4} + \ldots 
\]
Therefore, a simple calculation shows 
\begin{align*}
(-1)^p\diff{}{\xi}R^-_{p+1}(\xi) &= (-1)^p\diff{}{\xi}[P_{p+1} - P_{p}] \\
&= (-1)^p[(2p+1)P_p + (2p-3)P_{p-2} + \ldots - (2p-1)P_{p-1} - (2p-5)P_{p-3} - \ldots] \\
&= \sum_{k=0}^p (-1)^k(2k+1)P_k,
\end{align*}
which completes the proof.
\end{proof}

Using this proposition, we rewrite \eqref{eq:coef2} compactly as 
\begin{equation}\label{eq:coef3}
\pdiff{}{t} U_j + \frac{2a}{h_j}\pdiff{}{\xi}U_j =  (-1)^{p+1}\frac{a}{h_j}[[U_{j}]]\diff{}{\xi}R^-_{p+1}(\xi).
\end{equation}
This is an equation which the polynomial approximation $U_j$ will satisfy exactly on the cell $I_j$. In particular, $U_j$ solves the same advection equation as the exact solution $u$, except with a forcing term. Note that when approximating smooth solutions the polynomial approximation $U_j$ will be locally order $p+1$ accurate to the exact solution and thus the jump term at the cell interface $[[U_{j}]]$ will be of order $h_j^{p+1}$ as well. This implies that the solution to \eqref{eq:coef3} and, hence, the numerical approximation will in a sense be close to a solution of the advection equation since the forcing term is small.

\section{Fourier Analysis}

To investigate the properties of solutions of \eqref{eq:coef3} we look at a single Fourier mode solution of the form $U_j(\xi,t) = \hat{U}_j(\xi, \omega)e^{-a \omega t}$, where $\hat{U}_j$ is a polynomial of degree $p$ in $\xi$. As is customary in Fourier analysis, we also assume periodic boundary conditions on $I$. Using these assumptions on the form of $U_j(\xi,t)$ in \eqref{eq:coef3} we have that this Fourier mode satisfies the ODE
\begin{equation} \label{eq:fouriermode_ode}
-a\omega \hat{U}_j +  \frac{2a}{h_j}\pdiff{}{\xi}\hat{U}_j =  (-1)^{p+1}\frac{a}{h_j}[[\hat{U}_{j}]]\diff{}{\xi}R^-_{p+1}(\xi).
\end{equation}
We can solve this ODE explicitly to obtain that
\begin{equation}\label{eq:general_solution}
\hat{U}_j(\xi,\omega) = \hat{U}_j(-1,\omega)e^{\frac{\omega h_j}{2}(\xi+1)} + \frac{(-1)^{p+1}}{2}[[\hat{U}_{j}]]\int_{-1}^\xi e^{\frac{\omega h_j}{2}(\xi-s)}\diff{}{s}R^-_{p+1}(s) \; ds.
\end{equation}
The general solution \eqref{eq:general_solution} is not necessarily a polynomial in $\xi$. Therefore, \eqref{eq:general_solution} is too general for our purposes. Below, we will look for additional restrictions which ensure that this solution is polynomial in $\xi$. We state two lemmas which will help us to rewrite and investigate the integral term in \eqref{eq:general_solution}.

\begin{lemma} The integral term in \eqref{eq:general_solution} satisfies the following relation 
\begin{equation}\label{eq:Lemma1}
\frac{(-1)^{p+1}}{2}\int_{-1}^\xi e^{\frac{\omega h_j}{2}(\xi-s)}\diff{}{s}R^-_{p+1}(s) \; ds = -e^{\frac{\omega h_j}{2}(\xi+1)} + \frac{1}{(\omega h_j)^{p+1}}\left(g(\omega h_j)e^{\frac{\omega h_j}{2}(\xi+1)} - f(\omega h_j, \xi)\right),
\end{equation}
where $g(\omega h_j)$ is a polynomial of degree $p$ in $\omega h_j$ and $f(\omega h_j,\xi)$ is a polynomial of degree $p$ in both $\omega h_j$ and $\xi$. 
\end{lemma}

\begin{proof}
We begin by integrating the integral in \eqref{eq:general_solution} by parts
\begin{multline*}
\frac{(-1)^{p+1}}{2}\int_{-1}^\xi e^{\frac{\omega h_j}{2}(\xi-s)}\diff{}{s}R^-_{p+1}(s)\;ds= \frac{(-1)^{p+1}}{2} \left(\frac{2}{\omega h_j}\right) \left[ \diff{}{\xi}R^-_{p+1}(-1)e^{\frac{\omega h_j}{2}(\xi+1)} -\diff{}{\xi}R^-_{p+1}(\xi)\right] \\ + \frac{(-1)^{p+1}}{2} \left(\frac{2}{\omega h_j}\right)\int_{-1}^\xi e^{\frac{\omega h_j}{2}(\xi-s)}\diff{^2}{s^2}R^-_{p+1}(s)\;ds,
\end{multline*}
and then continue integrating by parts until the remaining integral vanishes to obtain
\begin{equation}\label{eq:lemma1_1}
\frac{(-1)^{p+1}}{2}\int_{-1}^\xi e^{\frac{\omega h_j}{2}(\xi-s)}\diff{}{s}R^-_{p+1}(s)\;ds=  \frac{1}{(\omega h_j)^{p+1}} \left( \tilde{g}(\omega h_j)e^{\frac{\omega h_j}{2}(\xi+1)} -f(\omega h_j,\xi)\right),
\end{equation}
where
\begin{equation}\label{eq:g_polynomial}
\tilde{g}(\omega h_j) = \frac{(-1)^{p+1}}{2} \sum_{k=1}^{p+1} 2^k (\omega h_j)^{p+1-k}\diff{^k}{\xi^k}R^-_{p+1}(-1),
\end{equation}
\begin{equation}\label{eq:f_polynomial}
f(\omega h_j,\xi) = \frac{(-1)^{p+1}}{2} \sum_{k=1}^{p+1} 2^k (\omega h_j)^{p+1-k}\diff{^k}{\xi^k}R^-_{p+1}(\xi). 
\end{equation}
Note that $\tilde{g}(\omega h_j)$ is a polynomial of degree $p$ in $\omega h_j$, and $f(\omega h_j,\xi)$ is polynomial of degree $p$ in $\omega h_j$ and $\xi$. Therefore, defining $g(\omega h_j) = (\omega h_j)^{p+1} + \tilde{g}(\omega h_j)$ in \eqref{eq:lemma1_1} we will obtain equation \eqref{eq:Lemma1} which completes the proof.
\end{proof}

\begin{lemma} The integral term in \eqref{eq:general_solution} also satisfies 
\begin{multline} \label{eq:Lemma2}
\frac{(-1)^{p+1}}{2}\int_{-1}^\xi e^{\frac{\omega h_j}{2}(\xi-s)}\diff{}{s}R^-_{p+1}(s) \; ds = -e^{\frac{\omega h_j}{2}(\xi+1)} + \frac{(-1)^{p+1}}{2}R^-_{p+1}(\xi) \\ + \frac{(-1)^{p+1}}{2} \sum_{k=1}^\infty \left(\frac{\omega h_j}{2}\right)^k\frac{1}{(k-1)!}\int_{-1}^\xi (\xi-s)^{k-1} R^-_{p+1}(s)\;ds.
\end{multline}
\end{lemma}

\begin{proof}
We prove this lemma in a similar manner to Lemma 1, i.e., by integrating the integral in \eqref{eq:general_solution} by parts, this time in reverse order
\begin{multline*}
\frac{(-1)^{p+1}}{2}\int_{-1}^\xi e^{\frac{\omega h_j}{2}(\xi-s)}\diff{}{s}R^-_{p+1}(s)\;ds= -e^{\frac{\omega h_j}{2}(\xi+1)} + \frac{(-1)^{p+1}}{2} R^-_{p+1}(\xi)  \\ + \frac{(-1)^{p+1}}{2} \frac{\omega h_j}{2} \int_{-1}^\xi e^{\frac{\omega h_j}{2}(\xi-s)} R^-_{p+1}(s)\;ds.
\end{multline*}
Here, from the definition of the Radau polynomial $R_{p+1}^-$, we have used $R_{p+1}^-(-1) = 2(-1)^p $. We continue integrating by parts to obtain
\begin{multline}
\frac{(-1)^{p+1}}{2}\int_{-1}^\xi e^{\frac{\omega h_j}{2}(\xi-s)}\diff{}{s}R^-_{p+1}(s)\;ds= -e^{\frac{\omega h_j}{2}(\xi+1)} + \frac{(-1)^{p+1}}{2} R^-_{p+1}(\xi) +  \frac{(-1)^{p+1}}{2} \left(\frac{\omega h_j}{2}\right) R^{-,(-1)}_{p+1}(\xi)  \\ +  \frac{(-1)^{p+1}}{2} \left(\frac{\omega h_j}{2}\right)^2 R^{-,(-2)}_{p+1}(\xi) + \ldots, \label{eq:lemma2}
\end{multline}
where we define $R^{-,(-k)}_{p+1}$ to be the repeated integrals of the right Radau polynomial, i.e., $R^{-,(0)}_{p+1}(\xi) = R^{-}_{p+1}(\xi)$ and
\begin{equation}\label{eq:Radau_integrals}
R^{-,(-(k+1))}_{p+1}(\xi) = \int_{-1}^\xi  R^{-,(-k)}_{p+1}(s)\; ds.
\end{equation}
Finally, using this definition with the Cauchy integration formula we can write the polynomials $R^{-,(-k)}_{p+1}$ as
\begin{equation}\label{eq:RadauIntegrals}
R^{-,(-k)}_{p+1}(\xi) = \frac{1}{(k-1)!}\int_{-1}^\xi (\xi-s)^{k-1} R^{-}_{p+1}(s)\; ds,
\end{equation}
which, when used in \eqref{eq:lemma2}, yields \eqref{eq:Lemma2} and completes the proof. 
\end{proof}

\noindent From these two lemmas we can establish a useful result regarding the polynomials $g$ and $f$. 

\begin{cor}
\begin{equation}\label{eq:Corollary1}
\frac{f(\omega h_j,\xi)}{g(\omega h_j)} = e^{\frac{\omega h_j}{2}(\xi+1)} - \frac{(-1)^{p+1} (\omega h_j)^{p+1}}{2g(\omega h_j)} \left[ R^-_{p+1}(\xi) +   \sum_{k=1}^\infty \left(\frac{\omega h_j}{2}\right)^k R^{-,(-k)}_{p+1}(\xi)  \right]
\end{equation}
and, in particular,
\begin{equation}\label{eq:Corollary2}
\frac{f(\omega h_j,1)}{g(\omega h_j)} = e^{\omega h_j} + \mathcal{O}((\omega h_j)^{2p+2}),
\end{equation}
i.e. $\frac{f(z,1)}{g(z)}$ is the $\frac{p}{p+1}$ Pad\'e approximant of $e^z$. 
\end{cor}
\noindent Before we state the proof of this corollary let us briefly recall the definition of a Pad\'e approximant \cite{Pade}. 
\begin{defi}
Given integers $m$ and $n$ and a sufficiently smooth function $F(z)$, the $\frac{m}{n}$ Pad\'e approximant of $F(z)$ is a rational function $\frac{P(z)}{Q(z)}$ where $P(z)$ and $Q(z)$ are polynomials of degree $m$ and $n$, respectively, and satisfy 
\[
\frac{P(z)}{Q(z)} = F(z) + \mathcal{O}(z^{m+n+1}).
\]  
This Pad\'e approximant is unique up to a constant multiple of the numerator and denominator. It is conventional to take $Q(0) = 1$ so that the Pad\'e approximant is uniquely defined. 
\end{defi}
\noindent We now proceed to prove the Corollary. 
\begin{proof}[Proof of Corollary 1]
Equating the right hand sides of \eqref{eq:Lemma1} and \eqref{eq:Lemma2} we obtain
\begin{multline*}
-e^{\frac{\omega h_j}{2}(\xi+1)} + \frac{1}{(\omega h_j)^{p+1}}\left(g(\omega h_j)e^{\frac{\omega h_j}{2}(\xi+1)} - f(\omega h_j, \xi)\right) = -e^{\frac{\omega h_j}{2}(\xi+1)} \\ + \frac{(-1)^{p+1}}{2} \left[ R^-_{p+1}(\xi) + \sum_{k=1}^\infty \left(\frac{\omega h_j}{2}\right)^k R^{-,(-k)}_{p+1}(\xi) \right].
\end{multline*}
Solving this expression for $\frac{f(\omega h_j,\xi)}{g(\omega h_j)}$ immediately yields \eqref{eq:Corollary1}. Subsequently, evaluating \eqref{eq:Corollary1} at $\xi =1$ we obtain
\[
\frac{f(\omega h_j,1)}{g(\omega h_j)} = e^{\omega h_j} - \frac{(-1)^{p+1} (\omega h_j)^{p+1}}{2g(\omega h_j)} \left[ R^-_{p+1}(1) +   \sum_{k=1}^\infty \left(\frac{\omega h_j}{2}\right)^k R^{-,(-k)}_{p+1}(1)  \right].
\]
From the definition of $R^{-}_{p+1}$ in terms of the Legendre polynomials we have that $R^{-}_{p+1} (1) = 0$. Furthermore, from the definition of $R^{-,(-k)}_{p+1}(\xi)$ in \eqref{eq:RadauIntegrals} and the orthogonality of $R^{-}_{p+1}$ to all polynomials of degree $p-1$ we have that $R^{-,(-k)}_{p+1}(1) = 0$ for $k=1,\ldots, p$. We therefore find that 
\[
\frac{f(\omega h_j,1)}{g(\omega h_j)} = e^{\omega h_j} + \frac{(-1)^{p+1} (\omega h_j)^{p+1}}{2g(\omega h_j)} \left[\sum_{k=p+1}^\infty \left(\frac{\omega h_j}{2}\right)^k R^{-,(-k)}_{p+1}(1)  \right],
\]
which yields
\[
\frac{f(\omega h_j,1)}{g(\omega h_j)} =  e^{\omega h_j} + \mathcal{O}((\omega h_j)^{2p+2}).
\]
By Lemma 1, $f(z,1)$ is a polynomial of degree $p$ while $g(z)$ is a polynomial of degree $p+1$ with the form $g(z) = z^{p+1} + \tilde{g}(z)$. Therefore, we find after a possible rescaling that the rational function $\frac{f(z,1)}{g(z)}$ approximates $e^{z}$ to order $2p+2$. Therefore it is the unique $\frac{p}{p+1}$ Pad\'e approximant of $e^{z}$.  
\end{proof}

\noindent Using Lemmas 1 and 2 we can write the general solution \eqref{eq:general_solution} in two ways:  
\begin{align}
\hat{U}_j(\xi,\omega) =&\; \hat{U}_j(-1,\omega)e^{\frac{\omega h_j}{2}(\xi+1)} - [[\hat{U}_{j}]]e^{\frac{\omega h_j}{2}(\xi+1)} + \frac{[[\hat{U}_{j}]]}{(\omega h_j)^{p+1}}\left( g(\omega h_j)e^{\frac{\omega h_j}{2}(\xi+1)} - f(\omega h_j, \xi) \right) \nonumber \\ 
=&\; \hat{U}_{j-1}(1,\omega)e^{\frac{\omega h_j}{2}(\xi+1)} + \frac{[[\hat{U}_{j}]]}{(\omega h_j)^{p+1}}\left( g(\omega h_j)e^{\frac{\omega h_j}{2}(\xi+1)} - f(\omega h_j, \xi) \right), \label{eq:fourier_node}
\end{align}
and
\begin{align}
\hat{U}_j(\xi,\omega) =&\; \hat{U}_j(-1,\omega)e^{\frac{\omega h_j}{2}(\xi+1)} - [[\hat{U}_{j}]]e^{\frac{\omega h_j}{2}(\xi+1)} + \frac{(-1)^{p+1}}{2}[[\hat{U}_{j}]]\left[ R^-_{p+1}(\xi) + \sum_{k=1}^\infty \left(\frac{\omega h_j}{2}\right)^k R^{-,(-k)}_{p+1}(\xi) \right]  \nonumber\\
 =&\; \hat{U}_{j-1}(1,\omega)e^{\frac{\omega h_j}{2}(\xi+1)} + \frac{(-1)^{p+1}}{2}[[\hat{U}_{j}]]\left[ R^-_{p+1}(\xi) + \sum_{k=1}^\infty \left(\frac{\omega h_j}{2}\right)^k R^{-,(-k)}_{p+1}(\xi) \right]. \label{eq:fourier_node2}
\end{align}
Note that the solution corresponding to the exact advection of the downwind point $\hat{U}_{j-1}(1,\omega)$ is $\hat{U}_{j-1}(1,\omega)e^{\frac{\omega h_j}{2}(\xi+1)}$ and, hence, from \eqref{eq:fourier_node} and \eqref{eq:fourier_node2} the general solution for the numerical approximation in cell $I_j$ consists of two parts: exact advection of the downwind value in cell $I_{j-1}$ and higher order error terms which are proportional to the magnitude of the jump at that interface. 

As mentioned above, we are interested only in polynomial solutions of \eqref{eq:fouriermode_ode}. The reason for this is that we know the numerical solution is polynomial in $\xi$ for all times $t$. Hence, the numerical solution should be composed solely of solutions of \eqref{eq:fouriermode_ode} which are polynomials in $\xi$. By examining \eqref{eq:fourier_node} we see that the solutions $U_j$ will be polynomial in $\xi$ only when
\begin{equation}\label{eq:poly_condition}
\hat{U}_{j-1}(1,\omega) + [[\hat{U}_j]]\frac{g(\omega h_j)}{(\omega h_j)^{p+1}} = 0
\end{equation}
is satisfied. Hence, assuming $g(\omega h_j) \neq 0$, we obtain after rearranging that $\hat{U}_j(-1,\omega)$ is related to $\hat{U}_{j-1}(1,\omega)$ by
\[
\hat{U}_j(-1,\omega) = \hat{U}_{j-1}(1,\omega)\frac{g(\omega h_j) - (\omega h_j)^{p+1}}{g(\omega h_j)}.
\]
Using the above relation in \eqref{eq:fourier_node} we obtain after rearranging that the polynomial solutions of \eqref{eq:fouriermode_ode} have the form
\begin{equation}\label{eq:particular_solution}
\hat{U}_j(\xi,\omega) = \hat{U}_{j-1}(1,\omega)\frac{f(\omega h_j, \xi)}{g(\omega h_j)}. 
\end{equation}
Thus, we obtain that the polynomial solutions on each cell are completely determined by the rational function $\frac{f(\omega h_j,\xi)}{g(\omega h_j)}$ and the value of the numerical solution at the downwind point of the previous cell. Using \eqref{eq:Corollary2} from Corollary 1 we have that at $\xi =1$ the polynomial solutions \eqref{eq:particular_solution} satisfy 
\begin{align}
\hat{U}_j(1,\omega) &= \hat{U}_{j-1}(1,\omega)\frac{f(\omega h_j, 1)}{g(\omega h_j)}\label{eq:particular_solution_downwind}  \\
 &= \hat{U}_{j-1}(1,\omega)e^{\omega h_j} + \mathcal{O}((\omega h_j)^{2p+2}).\nonumber 
\end{align}
\eqref{eq:particular_solution_downwind} means that the value of the solution at the right downwind point of cell $I_{j-1}$ is advected to the downwind point in the cell $I_j$ with local accuracy $2p+2$. This is related to the observed strong order $2p+1$ superconvergence at the downwind point of the numerical solution \cite{adjerid01a}. In fact, this is also related to the study of the superaccurate errors in disipation and dispersion of the DG scheme. The same Pad\'e approximant was studied by Hu and Atkins in \cite{HuAtkins02}, Ainsworth in \cite{Ainsworth04}, and Krivodonova and Qin in \cite{KQ13}. In each paper the authors note that the superaccuracies in dissipation and dispersion errors stem from the accuracy of this Pad\'e approximant. A key difference here, however, is that we have not made the assumption of a uniform mesh. Hence we immediately obtain from \eqref{eq:particular_solution_downwind} that we can extend the previously known results concerning the local $2p+2$ order of accuracy in dissipation and $2p+3$ order of accuracy in dispersion of the DG method to non-uniform meshes. 

In \cite{HuAtkins02,Ainsworth04} the Pad\'e approximant in \eqref{eq:particular_solution_downwind} was used to relate the numerical frequencies and wavenumbers of the scheme by assuming that $\omega$ was an exact frequency and finding the numerical wave number $\tilde{k}_n$. Another approach was taken in \cite{KQ13} where the authors were interested in the spectrum of the DG method, i.e. the precise values of the numerical frequencies $\omega$ with a fixed wavenumber $k$. In this work we will use the latter approach of estimating the numerical frequencies $\omega$ for problems with periodic boundary conditions. To this end we use the first line of \eqref{eq:particular_solution_downwind} together with periodicity of the boundary conditions to obtain the following condition on $\omega$ 
\begin{equation}\label{eq:freq_condition}
\prod_{j=1}^N \frac{f(\omega h_j, 1)}{g(\omega h_j)} = 1 .
\end{equation}
Hence, the numerical frequency $\omega$ must satisfy this relation. Solving \eqref{eq:freq_condition} for every value of $\omega$, however, is difficult since it would require finding the roots of a high order polynomial. An attempt to describe these values for some particular meshes was made in \cite{KQ13_b}, but this is beyond the scope of this paper. We will instead make the simplifying assumption that the mesh is uniform and obtain that the values of $\omega$ are solutions of 
\begin{equation}\label{eq:freq_condition_uniform}
\frac{f(\omega h, 1)}{g(\omega h)} = e^{\kappa_n h},
\end{equation}
where $e^{k_n h}$ is an $N$-th root of unity, i.e. $\kappa_n = \frac{2\pi n i}{L}$, $n = 0,\ldots, N-1$, where $L$ is the length of the domain $I$. Note that since the mesh is uniform and each downwind point of this solution is related by $\hat{U}_j(1,\omega) = \frac{f(\omega h, 1)}{g(\omega h)}\hat{U}_{j-1}(1,\omega) = e^{\kappa_n h}\hat{U}_{j-1}(1,\omega)$ the exact physical frequency for this wave is $\omega = \kappa_n$. In the following theorem we give an estimate on the values of $\omega$ which are solutions of \eqref{eq:freq_condition_uniform}.

\begin{theorem}
For a uniform mesh of $N$ elements, there are $(p+1)N$ solutions of \eqref{eq:fouriermode_ode} of the form $U_j(\xi,t) = \hat{U}_j(\xi,\omega)e^{-a\omega t} $ which are polynomials in $\xi$. These solutions satisfy $\hat{U}_{j}(1,\omega) = e^{\kappa_n h} \hat{U}_{j-1}(1,\omega)$ for each $j$ where $\kappa_n = \frac{2\pi n i}{L}$, $n = 0,\ldots, N-1$ . For each $\kappa_n$ there are $p+1$ values $\omega = \omega_{0}, \omega_{1}, \ldots, \omega_{p}$ which have the expansions
\[
\omega_{0} = \kappa_n + \mathcal{O}(\kappa_n^{2p+2}h^{2p+1})
\]
and 
\[
\omega_{m} = \frac{\mu_m}{h} + \mathcal{O}(\kappa_n), \qquad m = 1,\ldots, p,
\]
where $\mu_m$ are the $p$ non-zero roots of the polynomial $g(z) - f(z)$ and satisfy $\mathrm{Re}(\mu_m) > 0$. 
\end{theorem} 

\begin{proof}
We begin by noting that from \eqref{eq:Corollary2} 
\[
\frac{f(\omega h,1)}{g(\omega h)} = e^{\omega h} + \mathcal{O}((\omega h)^{2p+2}),
\] 
there should be at least one solution of \eqref{eq:freq_condition_uniform} of the form $\omega = \kappa_n + \mathcal{O}(\kappa_n^{2p+2}h^{2p+1})$. The condition \eqref{eq:freq_condition_uniform} itself for the numerical frequency $\omega$ can be rearranged to obtain
\begin{equation}
g(\omega h)e^{\kappa_n h} - f(\omega h,1) = 0.
\label{eq:lemma3_1}
\end{equation}
This expression is a polynomial of degree $p+1$ in $\omega$ and, therefore, has up to $p+1$ distinct roots. Regarding $h$ as a small parameter, we have from the form of \eqref{eq:lemma3_1} that we can asymptotically approximate each root using the expansion
\[
\omega = \frac{d_{-1}}{h} + d_0 + d_1 h + \ldots.
\]
Using this expansion in \eqref{eq:lemma3_1}, and expanding $e^{\kappa_n h} = 1 + \kappa_n h + \mathcal{O}((\kappa_n h)^2)$ we obtain
\begin{equation}\label{eq:lemma3_2}
g(d_{-1})-f(d_{-1},1) + g(d_{-1})\kappa_n h  + g'(d_{-1})d_0 h - f'(d_{-1},1)d_0 h +  \mathcal{O}((\kappa_n h)^2) = 0.
\end{equation} 
Setting the powers of $h$ equal to zero we find that we can determine the leading order asymptotic behaviour of each root by finding the possible values of the coefficient $d_{-1}$ which solve
\begin{equation}\label{eq:lemma3_3}
g(d_{-1})-f(d_{-1},1).
\end{equation}
Firstly, evaluating \eqref{eq:Corollary2} at $\omega h_j =0$ gives that $f(0,1) = g(0)$ so $d_{-1} = 0$ is a root of \eqref{eq:lemma3_3}. Furthermore, differentiating \eqref{eq:Corollary2} and evaluating at $\omega h_j =0$ yields that $f'(0,1) \neq g'(0)$ and, hence, $d_{-1} = 0$ is a simple root. Finally, when this Pad\'e approximant was studied in \cite{KQ13}, the authors showed that non-zero roots of the polynomial $g(z) - f(z,1)$ lay in the right-half complex plane. Therefore we can conclude that there are $p$ roots of the form 
\[
\omega_{m} = \frac{\mu_m}{h} + \mathcal{O}(\kappa_n),
\]
where $\mathrm{Re}(\mu_m) > 0$, and one root which corresponds to $d_{-1} = 0$.  Clearly the choice of $d_{-1} =0$ must correspond to the solution $\omega_{0} = \kappa_n + \mathcal{O}(\kappa_n^{2p+2}h^{2p+1})$, thus we obtain the result. 
\end{proof}
From this Theorem we have that for each $\kappa_n$ there are $p+1$ independent polynomial solutions of \eqref{eq:fouriermode_ode} which satisfy $\hat{U}_{j}(1,\omega) = e^{\kappa_n h}\hat{U}_{j-1}(1,\omega)$ for all $j$. One corresponds to $\omega_{0} = \kappa_n + \mathcal{O}(\kappa_n^{2p+2}h^{2p+1}) $ and can be seen as `physical' as it propagates with a numerical frequency which is close to the exact frequency. The other, `non-physical', solutions are dampened out exponentially quickly. This property of the numerical frequencies of the DG method was conjectured by Guo et al in \cite{guo2013} where the authors explicitly calculated similar expansions of the numerical frequencies $\omega_m$ for $p=1,2$ and 3. Since the non-physical modes are damped out exponentially quickly we see that  after sufficiently long times the accuracy of the numerical solution will be completely determined by the accuracy of the physical mode. Hence, if we specifically choose the initial projection of the exact solution to ensure that the physical mode is high-order accurate, we should preserve this accuracy for $t>0$. We formalize this observation in the following theorem. 

\begin{theorem}
Let $u(x,t)$ be a smooth exact solution of \eqref{eq:linear} on the interval $I$ with periodic boundary conditions. Let $U_h$ be the numerical solution of the DG scheme \eqref{eq:coef0} on a uniform mesh of $N$ elements and let $U_j$ be the restriction of the numerical solution to the cell $I_j$. Let $\epsilon_j(\xi,t) = U_{j} - u_j$ be the numerical error on $I_j$ (mapped to the canonical element $[ - 1,1]$). Suppose the projection of the initial profile $u(x,0)$ into the finite element space is chosen such that
\begin{equation}
\int_{-1}^1 \left[U_j(\xi,0) - u_j(\xi,0)\right]P_k(\xi) \; d\xi = \mathcal{O}(h^{2p+1-k}), \quad k = 0, \ldots, p, \label{eq:theorem_initial_condition}
\end{equation}
is satisfied. Then the error on cell $I_j$ will tend exponentially quickly towards the form
\begin{equation} \label{eq:SuperconvergenceForm}
\epsilon_j(\xi,t) = \frac{(-1)^{p+1}}{2} [[U_j]] R_{p+1}^- + \alpha_{p+2}(t) R_{p+1}^{-,(-1)} + \ldots + \alpha_{2p+1}(t)R_{p+1}^{-,(-p)} + \mathcal{O}(h^{2p+2}), 
\end{equation}
where $\alpha_k(t) = \mathcal{O}(h^{k})$
and, in particular, 
\[
\epsilon_j(1,t) = \mathcal{O}(h^{2p+1}).
\]
\end{theorem}

\begin{proof}
We begin by assuming for simplicity that the exact solution can be written as the sum
\[
u(x,t) = \sum_{n=0}^{N-1} \hat{u}_n e^{\kappa_n (x-at)},
\]
where $\kappa_n = \frac{2\pi ni}{L}$ and $L$ is the length of $I$. The coefficients $\hat{u}_n$ are found by the discrete Fourier transform and satisfy
\[
u(x_j,0) = \sum_{n=0}^{N-1} \hat{u}_n e^{\kappa_n x_j}, \quad j = 1, \ldots, N.
\]
Of course, in general the exact solution cannot be written in such a way but provided $u$ is sufficiently smooth and $N$ is sufficiently large the error in such an approximation should be negligible compared to the error in the polynomial approximation on each cell. Without loss of generality, let us consider the numerical approximation of just one of these Fourier modes, $u(x,t) = \hat{u}_n e^{\kappa_n (x-at)}$. Restricting this Fourier mode to the cell $I_j$ and mapping to the canonical element we see that 
\[
u_j(\xi, t) = \hat{u}_n e^{\kappa_n (x_j - at)} e^{\frac{\kappa_n h}{2}(\xi + 1)}.
\]
Since the mesh is uniform, the projection of $u_j(x,0)$ into the finite element space will be of the form $U_j(\xi,0) = \hat{u}_n e^{\kappa_n x_j} \hat{U}(\xi)$ for every $j$ and we immediately obtain that $U_j(1,t) = e^{\kappa_n h} U_{j-1}(1,t)$ for every $j$. We therefore can express the numerical solution as a sum of the $p+1$ independent polynomial solutions found in Theorem 1 which satisfy $U_j(\xi,t) = \hat{U}_j(\xi,\omega_m)e^{-a\omega_m t}$ and $\hat{U}_j(1,\omega_m) =  e^{\kappa_n h} \hat{U}_{j-1}(1,\omega_m)$, where the $\omega_{m}$ are the $p+1$ distinct values which satisfy $\frac{f(\omega_{m}h,\xi)}{g(\omega_{m}h)} = e^{\kappa_nh}$. Hence
\[
U_j(\xi,t) = \sum_{m=0}^{p} C_{m} e^{\kappa_n x_j-a\omega_{m} t} \frac{f(\omega_{m} h,\xi)}{g(\omega_{m} h)}.
\] 
Since the physical frequency $\omega_0$ is an accurate approximation of the exact frequency $\kappa_n$ to order $\mathcal{O}(\kappa_n^{2p+2}h^{2p+1})$, we have by Corollary 1 the expansion 
\[
C_0\frac{f(\omega_0 h,\xi)}{g(\omega_0 h)} = C_0 e^{\frac{\kappa_n h}{2}(\xi+1)} + C_0\frac{(-1)^{p+1} (\omega_0 h)^{p+1}}{2g(\omega_0 h)} \left[ R^-_{p+1}(\xi) +   \sum_{k=1}^\infty \left(\frac{\omega_0 h}{2}\right)^k R^{-,(-k)}_{p+1}(\xi)  \right]+ \mathcal{O}((\kappa_n h)^{2p+2}).
\]
Therefore performing the initial projection and using this expansion together with the orthogonality of the Radau polynomial $R_{p+1}^-$ we find that
\begin{multline*}
\int_{-1}^1 \left[U_j(\xi,0) - u_j(\xi,0)\right]P_k \; d\xi = \int_{-1}^1 (C_{0} - \hat{u}_n)e^{\kappa_n x_j}e^{\frac{\kappa_n h}{2}(x+1)}P_k \; d\xi \\ + \sum_{m=1}^p  C_{m}e^{\kappa_n x_j} \int_{-1}^1 \frac{f(\omega_{m} h,\xi)}{g(\omega_{m} h)} P_k\; d\xi + \mathcal{O}(\kappa_n^{2p+2}h^{2p+1-k}).
\end{multline*}
Thus, the requirement of the initial projection to satisfy \eqref{eq:theorem_initial_condition} will be satisfied by the choice of $C_{0} = \hat{u}_n + \mathcal{O}(h^{2p+1})$ and $\sum_{m=1}^p C_{m}\frac{f(\omega_{m} h,\xi)}{g(\omega_{m} h)} = \gamma P_p$, where $\gamma = \mathcal{O}(h^{p+1})$. Hence, this initial projection yields a high-order accurate physical mode of the numerical solution. Finally, we know from Theorem 1 that $\omega_1,\ldots,\omega_p$ have positive real parts of order $\mathcal{O}\left(\frac{1}{h}\right)$. Hence these components of the solution are damped out exponentially quickly in time and the numerical solution tends to the form
\begin{equation}\label{eq:Physical_Solution}
U_j(\xi,t) = \hat{u}_n e^{\kappa_n x_j-a\omega_{0} t} \frac{f(\omega_{0} h,\xi)}{g(\omega_{0} h)} + \mathcal{O}(h^{2p+1}).
\end{equation}
We can write this solution in the form \eqref{eq:fourier_node2} to find that
\begin{align}
U_j(\xi,t) &= \hat{u}_n e^{\kappa_n x_j +\frac{\omega_{0} h}{2}(\xi+1) -a\omega_{0} t} + \frac{(-1)^{p+1}}{2} [[U_j]] \left[ R^-_{p+1}(\xi) +   \sum_{k=1}^\infty \left(\frac{\omega_{0} h}{2}\right)^k R^{-,(-k)}_{p+1}(\xi)  \right]  + \mathcal{O}(h^{2p+1}), \nonumber \\
&= u_j(\xi,t) + \frac{(-1)^{p+1}}{2} [[U_j]] \left[ R^-_{p+1}(\xi) +   \sum_{k=1}^\infty \left(\frac{\omega_0 h}{2}\right)^k R^{-,(-k)}_{p+1}(\xi)  \right]  + \mathcal{O}(h^{2p+1}). \nonumber
\end{align}
From this, we see from that the error for the numerical approximation has the form
\[
\epsilon_j(\xi, t) =  \frac{(-1)^{p+1}}{2} [[U_j]] \left[ R^-_{p+1}(\xi) +   \sum_{k=1}^\infty \left(\frac{\omega_{0} h}{2}\right)^k R^{-,(-k)}_{p+1}(\xi)  \right]  + \mathcal{O}(h^{2p+1}),
\]
and since this is true for any Fourier node $\hat{u}_n e^{\kappa_n x}$, we obtain the result by summing this expression over all Fourier modes. 
\end{proof}

From \eqref{eq:SuperconvergenceForm} we see that the numerical solution will be one order more accurate at points $\xi_0$ such that $R_{p+1}^-(\xi_0) =0$, i.e. the roots of the right Radau polynomial. This is refered to as the spatial superconvergence of the method. From the proof of Theorem 2 we see that in order for the numerical error to tend to a superconvergent form the key requirement of the initial projection is that it projects onto the physical mode with high order accuracy. For example, an initial projection which consists of simply interpolating the initial data at equidistant points will not satisfy the conditions of Theorem 2 and thus we do not observe superconvergence of the numerical solution at the downwind points at any time. 

Examining the superconvergent form \eqref{eq:SuperconvergenceForm} we establish some useful corollaries. First, we note that once the non-physical modes have been damped out the remaining physical modes will be advected with order $h^{2p+1}$ accuracy. Hence, for the physical modes, the DG method can be viewed as an order $2p+1$ scheme. Second, since the initial projection in Theorem 2 produces a high-order accurate physical mode, and due to the orthogonality properties of the Radau polynomials, we also obtain high-order accuracy of the moments of the numerical solution. We states the results formally below. 

\begin{cor}
The accumulation error of the superconvergent numerical solution \eqref{eq:Physical_Solution} is of order $2p+1$. That is, after sufficiently long time that the non-physical modes of the numerical solution have been damped out, the numerical solution satisfies
\[
|| U_{j+1}(\xi, t+ah) - U_j(\xi,t)|| = \mathcal{O}(h^{2p+1}).
\]
\end{cor}

\begin{cor}
The superconvergent form of the numerical solution \eqref{eq:SuperconvergenceForm} has the property that the $m$-th moment of the error is order $2p+1-m$, i.e.
\[
\int_{-1}^1 \epsilon_j(\xi,t) P_m \; d\xi = \mathcal{O}(h^{2p+1-m}).
\]
\end{cor}

In the next section we perform several numerical test to confirm the results of Theorem 2 and Corollaries 2 and 3.


\section{Numerical Examples} 
In this section we will perform several numerical experiments to confirm the superconvergence properties stated in the section above for the DG method for the linear advection equation. Specifically, we will confirm that on a uniform mesh the numerical solution of the DG method with a non-superconvergent initial projection will tend exponentially quickly towards the superconvergent form \eqref{eq:SuperconvergenceForm}. Moreover, we will show that when $t$ is sufficiently large the superconvergent numerical solution is advected at order $\mathcal{O}(h^{2p+1})$. Finally, we will show that the moments of the numerical error are also high order accurate after sufficiently long time $t$. 

Our numerical studies were done on the initial value problem
\begin{eqnarray}\label{eq:linear_num}
u_t + u_x &=& 0, \qquad  -1 \leq x < 1,\qquad t\geq 0, \\
u(x,0) &=& u_0(x), \nonumber \\
u(-1,t) &=& u(1,t),\nonumber
\end{eqnarray}
with
\begin{equation}\label{eq:sin1D}
u_0(x) = \sin 4\pi x.
\end{equation}
All test below are calculated using an RK-4 time-stepping scheme and a $CFL$ number of $\frac{0.15}{2p+1}$ to minimize the error incurred in time integration. 

\subsubsection*{Superconvergence from more general initial projections}
In the proof of Theorem 1 we showed that the non-physical waves are damped out like $e^{-\frac{a\mu_m t}{h}}$.  We therefore expect to observe that a numerical solution with an initial projection satisfying the conditions of Theorem 2, to have converged to the superconvergent form \eqref{eq:SuperconvergenceForm} when 
\[
e^{-\frac{a \mu_{min} t}{h}} = \mathcal{O}(h^{2p+1}),
\]
where $\mu_{min}$ is the non-physical numerical frequency with the smallest real part. Therefore, we expect that the numerical solution will be superconvergent when
\[
t = -\frac{2p+1}{a\mu_{min}} \mathcal{O}(h \log h).
\]
We can estimate the smallest real part of the non-physical numerical frequencies by explicitly calculating the roots of the polynomial $g(z) - f(z,1)$ and finding the root with the smallest non-zero real part. This calculation for $p=1,2,3$, and 4 yields  $\mu_{min} = 6, 3, 0.42, $ and 0.058, respectively. Therefore, we see that the smallest real part of the non-physical numerical frequencies is decreasing very rapidly as the order $p$ increases. Hence we expect that it will take significantly longer for the non-physical modes to be damped out as the order of the DG method increases. 
\begin{figure}[t!]
\centering 
\begin{subfigure}[h]{8.5cm} 
\includegraphics[height=7cm,keepaspectratio=true]{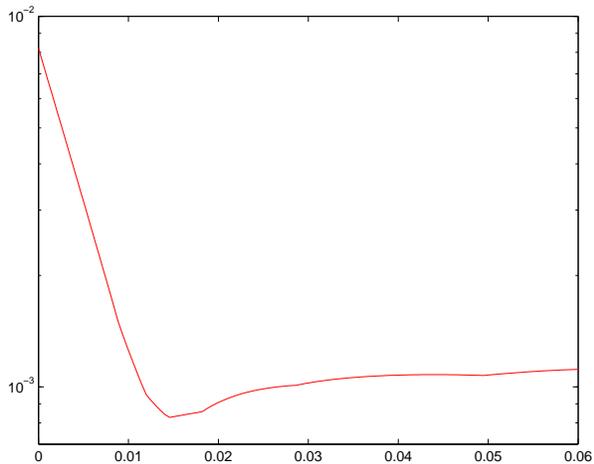}
\caption{$p=1$}
\end{subfigure} 
\quad
\begin{subfigure}[h]{8.5cm} 
\includegraphics[height=7cm,keepaspectratio=true]{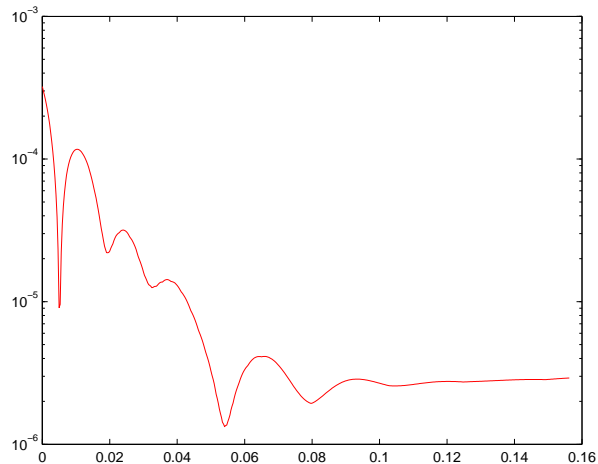}
\caption{$p=2$}
\end{subfigure}
\newline
\begin{subfigure}{8.5cm}
\includegraphics[height=7cm,keepaspectratio=true]{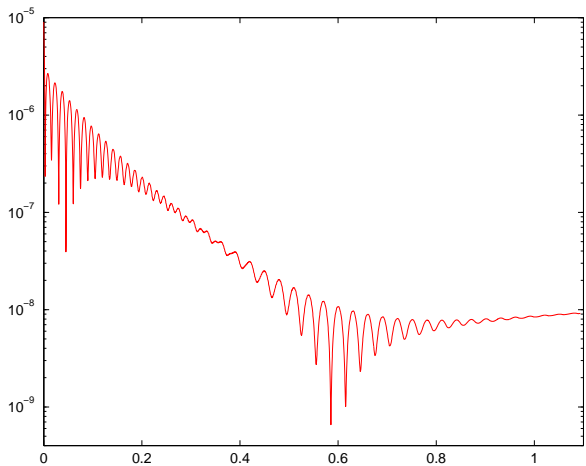} 
\caption{$p=3$}
\end{subfigure}
\caption{Semi-log plots of $L^1$ norm of the point-wise error at the downwind points of the numerical solution with $L^2$ initial projection as a function of time. Solutions are calculated for the linear advection, \eqref{eq:linear_num}-\eqref{eq:sin1D} on a uniform mesh of $N=64$ elements, $CFL = \frac{0.05}{2p+1}$.}
\label{fig:L1DownwindError}
\end{figure}

In Figure \ref{fig:L1DownwindError} we show the error at the downwind point of the numerical solution as a function of time for the $p=1,2,$ and 3 schemes on a uniform mesh of $N=64$ elements with the usual $L^2$ initial projection. The error at the downwind point is calculated using the $L^1$ norm of the point-wise numerical errors at the downwind points, i.e. $||E|| = h\sum_j |U_j(1,t) - u_j(1,t)|$. We notice from the linear shape of the semi-log plots that the error at the downwind point decays exponentially up to some critical time, at which point the error remains relatively constant. We also notice that due to the scaling of $\mu_{min}$ it takes significantly longer for the error at the downwind points to reach this critical time as $p$ increases. In the following numerical test we show that once this critical time is reached the error at the downwind points is $\mathcal{O}(h^{2p+1})$. 

\begin{table}[t]
\begin{center}
\begin{tabular}{|c|c|c|c|c|c|c|c|c|}\hline
& \multicolumn{4}{c|}{$L^2$ Projection} &\multicolumn{4}{c|}{Left Radau Projection}\\
\hline
$N$ &$||E||$ & $r$ & $||\bar{\epsilon}_0||$ & $r$ & $||E||$ & $r$ & $||\bar{\epsilon}_0||$ & $r$  \\
\hline\hline
16 &  7.02e-02 &    - & 6.66e-02 &    - & 9.63e-02 &    -  &  1.22e-01 &    -  \\
\hline
32 &  8.40e-03 & 3.06 & 8.90e-03 & 2.90 & 1.22e-02 & 2.98  &  1.68e-02 & 2.86 \\
\hline
64 &  1.04e-03 & 3.01 & 1.08e-03 & 3.04 & 1.54e-03 & 2.99  &  2.13e-03 & 2.98 \\
\hline
128 & 1.30e-04 & 3.00 & 1.34e-04 & 3.01 & 1.93e-04 & 2.99  &  2.67e-04 & 3.00 \\
\hline
256 & 1.63e-05 & 2.99 & 1.67e-05 & 3.00 & 2.43e-05 & 2.99  &  3.33e-05 & 3.00 \\
\hline
\end{tabular}
\end{center}
\caption{Linear advection, \eqref{eq:linear_num}-\eqref{eq:sin1D} with $p=1$ and with the $L^2$ and left Radau initial projections. $L^1$ error of the downwind points $||E||$ and of the cell averages $||\epsilon_0||$ are shown together with convergence rates, $r$. Errors are calculated at $t=h$.}
\label{ta:SinExactP1}
\end{table}

\begin{table}[t]
\begin{center}
\begin{tabular}{|c|c|c|c|c|c|c|c|c|}\hline
& \multicolumn{4}{c|}{$L^2$ Projection} &\multicolumn{4}{c|}{Left Radau Projection}\\
\hline
$N$ &$||E||$ & $r$ & $||\bar{\epsilon}_0||$ & $r$ & $||E||$ & $r$ & $||\bar{\epsilon}_0||$ & $r$  \\
\hline \hline
16 &  5.87e-03 &    - & 7.96e-03 &    - & 6.65e-03 &    - & 7.66e-03 &    -   \\
\hline
32 &  1.10e-04 & 5.72 & 1.86e-04 & 5.42 & 1.38e-04 & 5.59 & 2.20e-04 & 5.12  \\
\hline
64 &  2.74e-06 & 5.34 & 4.04e-06 & 5.52 & 3.57e-06 & 5.27 & 5.54e-06 & 5.31  \\
\hline
128 & 8.01e-08 & 5.10 & 1.10e-07 & 5.20 & 1.06e-07 & 5.07 & 1.60e-07 & 5.12  \\
\hline
256 & 2.47e-09 & 5.01 & 3.28e-09 & 5.07 & 3.31e-09 & 5.00 & 4.87e-09 & 5.03  \\
\hline
\end{tabular}
\end{center}
\caption{Linear advection, \eqref{eq:linear_num}-\eqref{eq:sin1D} with $p=2$ and with the $L^2$ and left Radau initial projections. $L^1$ error of the downwind points $||E||$ and of the cell averages $||\epsilon_0||$ are shown together with convergence rates, $r$. Errors are calculated at $t=4h$.}
\label{ta:SinExactP2}
\end{table}
\begin{table}[t]
\begin{center}
\begin{tabular}{|c|c|c|c|c|c|c|c|c|}\hline
& \multicolumn{4}{c|}{$L^2$ Projection} &\multicolumn{4}{c|}{Left Radau Projection}\\
\hline
$N$ &$||E||$ & $r$ & $||\bar{\epsilon}_0||$ & $r$ & $||E||$ & $r$ & $||\bar{\epsilon}_0||$ & $r$  \\
\hline\hline
16 &  5.14e-04 &    - & 1.05e-03 &    - & 5.14e-04 &    - & 1.06e-03 &    -   \\
\hline
32 &  2.36e-06 & 7.76 & 4.39e-06 & 7.90 & 2.30e-06 & 7.80 & 4.39e-06 & 7.91  \\
\hline
64 &  9.17e-09 & 8.00 & 1.77e-08 & 7.95 & 9.09e-09 & 7.99 & 1.82e-08 & 7.91  \\
\hline
128 & 3.63e-11 & 7.97 & 6.93e-11 & 8.00 & 3.53e-11 & 8.00 & 7.13e-11 & 8.00  \\
\hline
256 & 2.75e-13 & 7.05 & 6.53e-13 & 6.73 & 2.99e-13 & 6.88 & 5.83e-13 & 6.93  \\
\hline
\end{tabular}
\end{center}
\caption{Linear advection, \eqref{eq:linear_num}-\eqref{eq:sin1D} with $p=3$ and with the $L^2$ and left Radau initial projections. $L^1$ error of the downwind points $||E||$ and of the cell averages $||\epsilon_0||$ are shown together with convergence rates, $r$. Errors are calculated at $t=35h$.}
\label{ta:SinExactP3}
\end{table}

In Tables \ref{ta:SinExactP1}-\ref{ta:SinExactP3} we show the results of our convergence test for $p=1,2,$ and 3. In each table we present the $L^1$ errors at the downwind points of the cells $||E||$, and the $L^1$ norm of the numerical errors in the cell averages, calculated as
\[
||\bar{\epsilon}_0|| = h \sum_{j=1}^N \left| \int_{-1}^1 \left( U_j - u_j \right) \; d\xi \right| .
\]
The errors are calculated at $t = h, 4h,$ and $ 35h$ for the $p=1,2,$ and $3$ methods, respectively, in order to allow sufficient time for the non-physical modes to dampen out. We calculate these errors for two different initial projections.  The first is the usual $L^2$ projection while the second is a left Radau-like projection which is defined by 
\[
\int_{-1}^1 (U_j - u_j)P_k \; d\xi = 0, \quad k = 0, \ldots, p-1,
\]
and
\[
U_j(-1,0) = u_j(-1,0).
\]
These projections, while satisfying the conditions of Theorem 2, are far from the superconvergent form \eqref{eq:SuperconvergenceForm} which can be viewed as close to a right Radau projection of the exact solution. In each table we observe the expected $2p+1$ rate of convergence in both the error at the downwind points of the cells and in the cell averages.  

\subsubsection*{Order $2p+1$ advection of superconvergent solution}

\begin{table}[t]
\begin{center}
\begin{tabular}{|c|c|c|c|c|}
\hline
$N$ & $||U(x,0)-U(x,2)||$ & $r$ &$||U(x,2)-U(x,4)||$ & $r$  \\
\hline\hline
16 &  9.16e-03 &    - & 6.59e-03 &    -  \\
\hline
32 &  2.34e-03 & 1.96 & 8.34e-04 & 2.98  \\
\hline
64 &  5.90e-04 & 1.99 & 1.05e-04 & 3.00  \\
\hline
128 & 1.48e-04 & 2.00 & 1.31e-05 & 3.00  \\
\hline
\end{tabular}
\end{center}
\caption{Linear advection, \eqref{eq:linear_num}-\eqref{eq:sin1D} with $p=1$ and $L^2$ initial projection. $L^1$ norms of difference in numerical solutions at different times. Differences are measured between $U_j$ initially and at $t=2$, after one period, then between $U_j$ at $t=2$ and $t=4$, after an additional period.}
\label{ta:Physical_ModeP1}
\end{table}
\begin{table}[t]
\begin{center}
\begin{tabular}{|c|c|c|c|c|}
\hline
$N$ & $||U(x,0)-U(x,2)||$ & $r$ &$||U(x,2)-U(x,4)||$ & $r$  \\
\hline\hline
16 &  2.87e-04 &    - & 1.03e-05 &    -  \\
\hline
32 &  3.57e-05 & 3.00 & 3.24e-07 & 4.99  \\
\hline
64 &  4.46e-06 & 3.00 & 1.01e-08 & 5.00  \\
\hline
128 & 5.58e-07 & 3.00 & 3.17e-10 & 5.00  \\
\hline
\end{tabular}
\end{center}
\caption{Linear advection, \eqref{eq:linear_num}-\eqref{eq:sin1D} with $p=2$ and $L^2$ initial projection. $L^1$ norms of difference in numerical solutions at different times. Differences are measured between $U_j$ initially and at $t=2$, after one period, then between $U_j$ at $t=2$ and $t=4$, after an additional period.}
\label{ta:Physical_ModeP2}
\end{table}

Next, we show that once the non-physical modes of the numerical solution have been damped out, the remaining modes are advected at order $2p+1$. To show this we use the $L^2$ initial projection and calculate the norm of the difference between numerical solutions after $0,1$, and $2$ periods. That is, we calculate these differences as
\[
||U(x,0)-U(x,2)|| = h\sum_{j=1}^N \int_{-1}^{1} \left| U_j(\xi,0) - U_j(\xi,2)\right|\; d\xi.
\]
In Tables \ref{ta:Physical_ModeP1} and \ref{ta:Physical_ModeP2} we see that that the difference between the numerical solution initially and after one period converges at the usual $p+1$ rate. This is expected since the non-physical modes of the solution are present initially, and are $\mathcal{O}(h^{p+1})$. However, we also see that that the difference between the numerical solution after one and two periods converges with order $2p+1$. This shows that once the non-physical modes of the solution have been damped out, the remaining physical modes are advected at order $2p+1$. 

\subsubsection*{Superconvergence of moments}

\begin{table}[t]
\begin{center}
\begin{tabular}{|c|c|c|c|c|c|c|c|c|}\hline
& \multicolumn{4}{c|}{$L^2$ Projection} &\multicolumn{4}{c|}{Left Radau Projection}\\
\hline
$N$ &$||\bar{\epsilon}_{1}||$ & $r$ & $||\bar{\epsilon}_{2}||$ & $r$ & $||\bar{\epsilon}_{1}||$ & $r$ & $||\bar{\epsilon}_{2}||$ & $r$  \\
\hline\hline
16 &  2.92e-03 &    - & 8.27e-03 &    - & 3.24e-03 &    -  &  8.06e-03 &    -  \\
\hline
32 &  1.12e-04 & 4.70 & 1.04e-03 & 2.99 & 1.04e-04 & 4.96  &  1.04e-03 & 2.95 \\
\hline
64 &  8.09e-06 & 3.79 & 1.29e-04 & 3.01 & 7.97e-06 & 3.70  &  1.29e-04 & 3.01 \\
\hline
128 & 5.21e-07 & 3.96 & 1.61e-05 & 3.00 & 5.19e-07 & 3.94  &  1.61e-05 & 3.00 \\
\hline
256 & 3.28e-08 & 3.99 & 2.00e-06 & 3.00 & 3.27e-08 & 3.99  &  2.01e-06 & 3.00 \\
\hline
\end{tabular}
\end{center}
\caption{Linear advection, \eqref{eq:linear_num}-\eqref{eq:sin1D} with $p=2$ and with the $L^2$ and left Radau initial projections. $L^1$ norms of the first and second moments of the numerical error are shown together with convergence rates, $r$. Errors are calculated at $t=4h$.}
\label{ta:MomentsP2}
\end{table}

Finally, we demonstrate the high order accuracy in the moments of the numerical error for $p=2$ in Table \ref{ta:MomentsP2}. We present the $L^1$ norm of the first and second moments of the numerical error in each cell. The moments are calculated as
\[
||\bar{\epsilon}_{m}|| = h\sum_{j=1}^N \left| \int_{-1}^{1} \left( U_j - u_j\right) P_m \; d\xi \right|.
\] 
The moments are calculated from the numerical solution using the usual $L^2$ initial projection and the left Radau-like projection, as above. From this table we see that the $m$-th moment of the numerical error does indeed achieve the predicted order $2p +1 -m$ convergence rate. 

\section{Discussion}
We have shown that the numerical solution of the DG method is closely related to the $\frac{p}{p+1}$ Pad\'e approximant of the exponential function $e^z$. Indeed, by finding the Fourier modes of the PDE \eqref{eq:coef3} which governs the numerical solution we have shown that the polynomial solutions are completely described by the value of the solution at the downwind point of the previous cell and the rational function $\frac{f(\omega h_j,\xi)}{g(\omega h_j)}$. This rational function has a local expansion in $h_j$ in terms of the $(p+1)$-th right Radau polynomial and the anti-derivatives of this polynomial. Furthermore, at the downwind point of the cell we have that $\frac{f(\omega h_j,\xi)}{g(\omega h_j)}$ is the $\frac{p}{p+1}$ Pad\'e approximant of $e^z$. As discussed in \cite{HuAtkins02} and \cite{Ainsworth04}, the accuracy of this Pad\'e approximant is what gives rise to the high order accuracies in both dissipation and dispersion of the DG scheme known as superaccuracy. The expansion of this rational function in terms of the right Radau polynomial and its anti-derivatives is what we observe to be the superconvergence of the numerical solution at the right Radau points and the order $2p+1$ superconvergence of the downwind point in each cell. Finally, as studied in \cite{KQ13} and shown by equation \eqref{eq:freq_condition}, the spectrum of the DG discretization matrix is directly related to this $\frac{p}{p+1}$ Pad\'e approximant. Therefore, these Fourier modes provide a direct connection between the three seemingly disparate properties of superaccuracy, superconvergence, and the spectrum of the DG method.   

We have shown that for a uniform computational mesh of $N$ elements there exist $N$ polynomial solutions which can be viewed as physical components of the numerical wave and $pN$ polynomial solutions which are non-physical components. Moreover, these non-physical solutions are damped out exponentially quickly in time and, therefore, neglecting time integration errors we can conclude that the accuracy of the numerical solution for sufficiently large times is completely determined by the accuracy of the initial projection of the exact solution onto the physical modes. Furthermore, the DG scheme can be viewed as order $2p+1$ accurate on these physical solutions. Using this result we proved that for a class of initial projections of the exact solution we expect to obtain a numerical solution which is superconvergent at both the roots of the right Radau polynomial and the downwind points of the cell, after sufficiently long times. In particular, there is a class of initial projections which do not initially have order $h^{2p+1}$ accuracy at the downwind point, but will indeed obtain this order of accuracy after sufficient time has elapsed. For these projections the points of superconvergence will migrate to the roots of the right Radau polynomial exponentially quickly in time.    

We intend to investigate whether this approach would be useful in both identifying the superconvergence and superaccuracy properties of the DG scheme in non-linear and multidimensional problems, and in determining the spectral properties of the DG method in these problems.

\bibliographystyle{plain}
\bibliography{dg}

\end{document}